\documentclass[a4paper, reqno, 14pt]{amsart}

\usepackage[usenames,dvipsnames]{color}
\usepackage{amsthm,amsfonts,amssymb,amsmath,amsxtra, appendix}
\usepackage[all]{xy}
\SelectTips{cm}{}
\usepackage{xr-hyper}
\usepackage{verbatim}
\usepackage{mathrsfs}
\usepackage{diagbox}
\usepackage{enumerate}
\usepackage[shortlabels]{enumitem}
\usepackage{tikz}
\usepackage{tkz-euclide}
\usepackage{setspace} 
\usepackage{hyperref}

\newcommand{\remind}[1]{{\bf ** #1 **}}

\RequirePackage{xspace}
\RequirePackage{etoolbox}
\RequirePackage{varwidth}
\RequirePackage{enumitem}
\RequirePackage{tensor}
\RequirePackage{mathtools}
\RequirePackage{longtable}
\RequirePackage{multirow}

\usepackage{tikz}
\usepackage{tikz-cd}
\usepackage{csquotes}

\usepackage[url=false,doi=false,isbn=false,sorting=nyt,giveninits=true,maxbibnames=4]{biblatex}
\def\ge{\geqslant}
\def\le{\leqslant}
\def\a{\alpha}

\def\g{\gamma}

\def\d{\delta}
\def\D{\Delta}

\def\s{\sigma}
\def\t{\tau}
\def\th{\theta}

\def\l{\lambda}

\def \ud{\underline}

\def\af{\mathrm{af}}
\def\lup{\rightharpoonup}

\def\rank{\mathrm{rank}}

\def\<{\langle}
\def\>{\rangle}

\newcommand{\de}{{\mathrm{def}}}

\newcommand{\BA}{\ensuremath{\mathbb {A}}\xspace}

\newcommand{\BF}{\ensuremath{\mathbb {F}}\xspace}
\newcommand{{\BG}}{\ensuremath{\mathbb {G}}\xspace}

\newcommand{\BJ}{\ensuremath{\mathbb {J}}\xspace}
\newcommand{{\BK}}{\ensuremath{\mathbb {K}}\xspace}

\newcommand{\BN}{\ensuremath{\mathbb {N}}\xspace}

\newcommand{\BQ}{\ensuremath{\mathbb {Q}}\xspace}
\newcommand{\BR}{\ensuremath{\mathbb {R}}\xspace}
\newcommand{\BS}{\ensuremath{\mathbb {S}}\xspace}

\newcommand{\BZ}{\ensuremath{\mathbb {Z}}\xspace}

\newcommand{\CI}{\ensuremath{\mathcal {I}}\xspace}

\newcommand{\CK}{\ensuremath{\mathcal {K}}\xspace}

\newcommand{\CO}{\ensuremath{\mathcal {O}}\xspace}

\newcommand{\CS}{\ensuremath{\mathcal {S}}\xspace}
\newcommand{\CT}{\ensuremath{\mathcal {T}}\xspace}

\newcommand{\CW}{\ensuremath{\mathcal {W}}\xspace}

\newcommand{\Ad}{{\mathrm{Ad}}}

\newcommand{\cl}{{\mathrm{cl}}}

\DeclareMathOperator{\Gal}{Gal}

\newcommand{\id}{\ensuremath{\mathrm{id}}\xspace}

\def\tW{\tilde W}

\def\dw{\dot{w}}
\def\dx{\dot{x}}

\def\doubleparenthesis#1{{(\!({#1})\!)}}
\def\doublebracket#1{{[\![{#1}]\!]}}

\newtheorem{theorem}{Theorem}
\newtheorem{proposition}[theorem]{Proposition}
\newtheorem{lemma}[theorem]{Lemma}

\newtheorem{corollary}[theorem]{Corollary}

\theoremstyle{definition}
\newtheorem{definition}[theorem]{Definition}
\newtheorem{example}[theorem]{Example}
\newtheorem*{example*}{Example}

\newtheorem{remark}[theorem]{Remark}

\newtheorem*{function*}{Function}

\numberwithin{equation}{section}
\numberwithin{theorem}{section}

\setitemize[0]{leftmargin=*,itemsep=\the\smallskipamount}
\setenumerate[0]{leftmargin=*,itemsep=\the\smallskipamount}

\renewcommand{\to}{%
   \ifbool{@display}{\longrightarrow}{\rightarrow}%
   }
\let\shortmapsto\mapsto
\renewcommand{\mapsto}{%
   \ifbool{@display}{\longmapsto}{\shortmapsto}%
   }
\newlength{\olen}
\newlength{\ulen}
\newlength{\xlen}
\newcommand{\xra}[2][]{%
   \ifbool{@display}%
      {\settowidth{\olen}{$\overset{#2}{\longrightarrow}$}%
       \settowidth{\ulen}{$\underset{#1}{\longrightarrow}$}%
       \settowidth{\xlen}{$\xrightarrow[#1]{#2}$}%
       \ifdimgreater{\olen}{\xlen}%
          {\underset{#1}{\overset{#2}{\longrightarrow}}}%
          {\ifdimgreater{\ulen}{\xlen}%
             {\underset{#1}{\overset{#2}{\longrightarrow}}}
             {\xrightarrow[#1]{#2}}}}%
      {\xrightarrow[#1]{#2}}
   }
\makeatother
\newcommand{\xyra}[2][]{%
   \settowidth{\xlen}{$\xrightarrow[#1]{#2}$}%
   \ifbool{@display}%
      {\settowidth{\olen}{$\overset{#2}{\longrightarrow}$}%
       \settowidth{\ulen}{$\underset{#1}{\longrightarrow}$}%
       \ifdimgreater{\olen}{\xlen}%
          {\mathrel{\xymatrix@M=.12ex@C=3.2ex{\ar[r]^-{#2}_-{#1} &}}}%
          {\ifdimgreater{\ulen}{\xlen}%
             {\mathrel{\xymatrix@M=.12ex@C=3.2ex{\ar[r]^-{#2}_-{#1} &}}}
             {\mathrel{\xymatrix@M=.12ex@C=\the\xlen{\ar[r]^-{#2}_-{#1} &}}}}}%
      {\mathrel{\xymatrix@M=.12ex@C=\the\xlen{\ar[r]^-{#2}_-{#1} &}}}%
   }
\makeatletter
\newcommand{\xla}[2][]{%
   \ifbool{@display}%
      {\settowidth{\olen}{$\overset{#2}{\longleftarrow}$}%
       \settowidth{\ulen}{$\underset{#1}{\longleftarrow}$}%
       \settowidth{\xlen}{$\xleftarrow[#1]{#2}$}%
       \ifdimgreater{\olen}{\xlen}%
          {\underset{#1}{\overset{#2}{\longleftarrow}}}%
          {\ifdimgreater{\ulen}{\xlen}%
             {\underset{#1}{\overset{#2}{\longleftarrow}}}
             {\xleftarrow[#1]{#2}}}}%
      {\xleftarrow[#1]{#2}}
   }
\newcommand{\isoarrow}{%
   \ifbool{@display}{\overset{\sim}{\longrightarrow}}{\xrightarrow\sim}%
   }

\begin{document}

\title[]{Lifting Deligne-Lusztig Reduction and Geometric Coxeter Type Elements} 

\author{Sian Nie}
\address{Academy of Mathematics and Systems Science, Chinese Academy of Sciences, Beijing 100190, China}

\address{ School of Mathematical Sciences, University of Chinese Academy of Sciences, Chinese Academy of Sciences, Beijing 100049, China}
\email{niesian@amss.ac.cn}

\author{Felix Schremmer}
\address{Department of Mathematics and New Cornerstone Science Laboratory, The University of Hong Kong, Hong Kong SAR, China}

\email{schremmer@hku.hk}

\author[Qingchao Yu]{Qingchao Yu}
\address{Institute for Advanced Study, Shenzhen University, Nanshan District, Shenzhen, Guangdong, China}
\email{qingchao\_yu@outlook.com}

\thanks{}



\begin{abstract}
Cases of Shimura varieties where the special fibre of a Rapoport-Zink space is simply the union of classical Deligne-Lusztig varieties are known as \emph{fully Hodge-Newton decomposable} ones, and have been studied with great interest in the past. In recent times, the focus has shifted to identify tractable cases beyond the fully Hodge-Newton decomposable ones, and several instances have been identified where only products of classical Deligne-Lusztig varities with simpler spaces occur.

In our paper, we provide a uniform framework to capture these phenomena. By studying liftings from the affine flag variety to the loop group and combining them with the Deligne-Lusztig reduction method, our main result is a powerful criterion to show that an affine Deligne-Lusztig variety is the product of a classical Deligne-Lusztig variety with affine spaces and pointed affine spaces.

We introduce the class of elements that we call having \emph{geometric Coxeter type}, strictly including previously studied notions such as positive Coxeter type or finite Coxeter type. These elements of geometric Coxeter type satisfy the conditions for our main result and also a condition on the Newton stratification introduced by Mili\'cevi\'c-Viehmann.
\end{abstract}

\maketitle

\section{Introduction}
In a seminal paper, Deligne and Lusztig \cite{Deligne1976} constructed a class of varieties, henceforth known as Deligne-Lusztig varieties. Given a connected reductive group $G$ over $\mathbb F_q$ with Frobenius $\sigma$ and an element $w$ of the Weyl group of $G$, the Deligne-Lusztig variety $X_w$ is a reduced scheme over $\overline{\mathbb F}_q$ parametrizing Borel subgroups $B$ of $G_{\overline{\mathbb F}_q}$ such that $B$ and $\sigma(B)$ have relative position $w$. Then the group $G(\mathbb F_q)$ acts on $X_w$, and the \'etale cohomology of these Deligne-Lusztig varieties is of central interest for the representation theory of finite Groups of Lie type.

The affine analogues of Deligne-Lusztig varieties were introduced by Rapoport \cite{Ra} in the context of reduction of Shimura varieties. Since then, affine Deligne-Lusztig varieties (ADLV) have been studied intensely, not only for their connection to Shimura varieties but also to study moduli spaces of local $G$-shtukas and affine Lusztig varieties. In this affine situation, we are given a non-archimedian local field $F$.

Concretely, such a field would either be given by the field of formal Laurent series over a finite field $F = \mathbb F_q\doubleparenthesis t$, or as a finite extension of some $p$-adic rationals. Denote the completion of the maximal unramified extension by $L$, and the Frobenius again by $\sigma\in \Gal(L/F)$. We choose a $\sigma$-stable Iwahori subgroup $\CI\subseteq G(L)$. Then for any two elements $w,b\in G(L)$, the ADLV $X_w(b)$ is a reduced subscheme (or sub perfect scheme) of the affine flag variety with $\overline{\mathbb F}_q$-valued points
\begin{align*}
    X_w(b) = \{g\in G(L)/\CI \mid g^{-1} b\sigma(g) \in \CI w \CI\}.
\end{align*}
The group
\begin{align*}
\BJ_b(F) = \{ g \in G(L) \mid  g^{-1}b\s(g) = b\}
\end{align*}
acts naturally on $X_w(b)$ by left multiplication.

After decades of focused research, we have some \enquote{coarse} understanding of the geometry of $X_w(b)$, e.g.\ a natural criterion to determine if $X_w(b)$ is empty for basic $b$ \cite{GHN1} as well as an algorithm to determine the dimension of $X_w(b)$ together with the number of $\BJ_b(F)$-orbits of top dimensional irreducible components \cite[Theorem~2.19]{He-CDM}. However, a finer description of the geometry has only been found in a handful of special cases.

In \cite{SSY}, the second and third named author together with Shimada introduce a class of affine Weyl group elements, named elements of \emph{positive Coxeter type}. If $w\in\tW$ has positive Coxeter type, $b$ is basic and $G$ is split over $F$, then they show that $X_w(b)$ is universally homeomorphic, as scheme over $\overline{\mathbb F}_q$, to the product of a (classical) Deligne-Lusztig variety of Coxeter type, an affine space and a disjoint union of points. One key step in that proof is to construct a \enquote{lifting} $\psi: X_w(b)\to  G(L)$, which composed with the natural projection $G(L)\to G(L)/I$ yields the identity map on $X_w(b)$.

In this article, we study such liftings systematically. We show that they can be constructed inductively using a reduction method of Deligne-Lusztig. For a recall of the notion of reduction paths and trees, we refer to Section~\ref{sec:2.4}.

Our main result is the following.

\begin{theorem}[Theorem \ref{thm:main}]
Let $w\in\tW$, $b \in G(L)$ with $X_w(b)\ne\emptyset$. Let $\mathcal{T}$ be a reduction tree of $w$. Assume that
\begin{enumerate}
    \item there exists a unique path $\ud p$ of $\CT$ such that $\text{end}(\ud p) \in [b]$.
    \item $X_{\mathrm{end}(\ud p)}(b) $ has a lift.
\end{enumerate}
Then $X_w(b)$ is universally homeomorphic to the trivial fiber bundle
$$  X_{\mathrm{end}(\ud p)}(b) \times (\BG_m)^{\ell_{\text{I}}(\ud p)} \times  (\BA^1)  ^{\ell_{\text{II}}(\ud p)},$$
where $\ell_{\text{I}}(\ud p)$ and $\ell_{\text{II}}(\ud p)$ are the number of type I and type II reduction steps in $\ud p$ respectively.
\end{theorem}
The geometry of $X_{\mathrm{end}(\ud p)}(b)$ is well-understood, being essentially a union of classical Deligne-Lusztig varieties. We remark that universal homeomorphisms become isomorphisms in the category of perfect schemes, and that passing to perfection does not affect \'etale cohomology.

We say an element $w$ satisfies the strong multiplicity one property, if for any $[b]\in B(G)_w$ and some (equivalently any) reduction tree $\CT$ of $w$, there exists a unique path $\ud p$ of $\CT$ such that $\text{end}(\ud p) \in [b]$.

We say $w$ is of geometric Coxeter type, if $w$ satisfies the strong multiplicity one property and each endpoint of $w$ is a cyclic shift of an element of the form $ux$, where $x$ is a $\s$-straight and $u$ is a $\Ad(x)\circ\s$-Coxeter element (see Definition \ref{def:geo-cox} and Example \ref{ex:geo-cox}).

The notion of geometric Coxeter element generalizes the elements of the form $t^{\mu}c$, which have been studied extensively in \cite{HNY1}, and elements of positive Coxeter type from \cite{SSY}. 

As an application of our main result and a classical result of Lusztig \cite[Theorem 2.6]{Lusz76}, we prove the following.
\begin{theorem}[Theorem \ref{thm:geo-cox}]
Let $w\in \tW$ be a geometric Coxeter type element and $b\in G(L)$ with $X_w(b) \ne \emptyset$. Then 
\begin{enumerate}
    \item all irreducible components of $X_w(b)$ lie in a single $\BJ_b(F)$-orbit;

    \item each irreducible component of $X_w(b)$ is universally homeomorphic to the trivial bundle
$$X'\times (\BG_m)^{\ell_{\text{I}}(\ud p)} \times (\BA^1) ^{\ell_{\text{II}}(\ud p)},$$
where $X'$ is a classical Deligne-Lusztig variety of Coxeter type, $\ell_{\text{I}}(\ud p)$ and $\ell_{\text{II}}(\ud p)$ are as in Theorem \ref{thm:main}.
\end{enumerate}
\end{theorem}

Note that the statement (2) is new even for elements of the form $t^{\mu}c$.

Given an element $w\in\tW$ of geometric Coxeter type, we can also give a complete description of the set of all $b\in G(L)$ with $X_w(b)\neq\emptyset$, and describe the numbers $\ell_{\text{I}}(\ud p)$ and $\ell_{\text{II}}(\ud p)$ directly in terms of $w$ and $b$, without referring to a choice of reduction path. This is done in Section~\ref{sec:6}, using the purity of the Newton stratification.

Examples for elements of geometric Coxeter type often arise from the Ekedahl-Oort stratification of affine Deligne-Lusztig varieties in affine Grassmannians. Given a $\sigma$-stable hyperspecial $\CK\supseteq \CI$ and a cocharacter $\mu\in X_\ast(T)$ coming from the Shimura datum, the reduced special fibre of the corresponding Rapoport-Zink space is given by
\begin{align*}
    X_\mu(b) = \{g \CK\in G(L)/\CK\mid g^{-1} b\sigma(g)\in \CK\varepsilon^{\mu} \CK\}.
\end{align*}
Let $W_{\CK}\subseteq \tW$ be the finite subgroup generated by the affine reflections contained in $\CK$, and $w\in W_{\CK} t^\mu W_{\CK}$ be of minimal length in its right $W_{\CK}$ coset.
An Ekedahl-Oort stratum of $X_\mu(b)$ is given by
\begin{align*}
    \mathrm{EO}_{\CK,w}(b) := \{g\in G(L)/\CK\mid \exists k\in \CK:~ (gk)^{-1} b\sigma(gk)\in \CI w \CI\}\subseteq X_\mu(b).
\end{align*}
It is known from \cite{HeRapoport} that $X_\mu(b)$ is the disjoint union of these strata, and they also give a combinatorial description of the closure relations among the strata. Moreover, the map $X_w(b)\to \mathrm{EO}_{\CK,w}(b)$, sending $g \CI\in G(L)/\CI$ to $g \CK\in G(L)/ \CK$ is surjective with finite fibres \cite[Theorem~6.21]{HeRapoport}.

If $\mu$ is not central, then is always at least one Ekedahl-Oort stratum of the form $w = t^\mu c$, where $c$ is a $\sigma$-Coxeter element in the finite Weyl group. This element has geometric Coxeter type as explained above, and we may conjecture that the map $X_w(b)\to \mathrm{EO}_{K,w}(b)$ is an isomorphism in this case. Hence our main result applies to at least this one Ekedahl-Oort stratum for every Shimura datum.

In some particular cases of Shimura varieties, such as the variety $\mathrm{GU}(2,n-2)$ at split primes \cite{Shimada2024_gln} or the Siegel modular variety of genus $3$ \cite{Shimada2024_siegel}, it follows from these explicit studies made by Shimada that all Ekedahl-Oort strata are of positive Coxeter type. Further such cases will be discussed systematically in a forthcoming article by the second named author and Viehmann.

In \cite{CI}, Chan and Ivanov studied Lusztig semi-infinite Deligne-Lusztig varieties for inner forms of ${\rm GL}_n$, whose homology groups realize cases of local Langlands correspondence and Jaquet-Langlands correspondence. For certain special element of geometric Coxeter type, they proved that the attached semi-infinite Deligne-Lusztig variety is a natural inverse limit of corresponding deep level affine Deligne-Lusztig varieties, which are products of classical Deligne-Lusztig varieties and affine spaces. Similar results are also obtained for ${\rm GSp}_{2n}$ by Takamatsu \cite{Takamatsu}. By combining Theorem 1.2 and the work \cite{Nie2023}, it is likely to extended these results of Chan, Ivanov and Takamatsu to all elements of geometric Coxeter type in arbitrary Iwahori-Weyl groups. 

We review the main group theoretic setup and the Deligne-Lusztig reduction in Section~\ref{sec:2}. Special care is taken to make sure we can treat both the equal and mixed characteristic case simultaneously. Our main result is introduced and proved in Sections \ref{sec:3} and \ref{sec:4}. We introduce the notion of geometric Coxeter type in Section~\ref{sec:5}, where we also show their lifting property. In the final Section~\ref{sec:6}, we study the Newton stratification inside $\CI w \CI$ for $w$ of geometric Coxeter type following Mili\'cevi\'c-Viehmann. We show that the closure of a Newton stratum is a union of Newton strata, and derive simple combinatorial formulas for $\dim X_w(b)$ and $\ell_1(\ud p), \ell_2(\ud p)$ in Theorem~\ref{thm:geo-cox}.

\subsection*{Acknowledgments}
We thank Miaofen Chen, Ryosuke Shimada and Eva Viehmann for helpful discussions. SN was partially supported by National Key R\&D Program of China, No. 2020YFA0712600, CAS Project for Young Scientists in Basic Research, Grant No. YSBR-003, and National Natural Science Foundation of China, No. 11922119, No. 12288201 and No. 12231001. FS was partially supported by the New Cornerstone Foundation through the New Cornerstone Investigator grant, and by Hong Kong RGC grant 14300122, both awarded to Prof. Xuhua He.

\section{Preliminary}\label{sec:2}
\subsection{}\label{sec:2.1}
Let $F$ denote a non-archimedian local field with ring of integers $\CO_F$, whose completion of the maximal unramified extension we denote by $L$. This means that $F$ is either a finite extension of some $p$-adic field, or a field of formal Laurent series over a finite field. In either case, we let $\varepsilon\in \CO_F$ be a uniformizer and denote the Frobenius of $L/F$ by $\s\in \Gal(L/F)$. The residue field of $F$ is denoted $\mathbb F_q$.


Let $G$ be a connected and reductive group over $\CO_F$.

We fix a maximal torus and Borel $T\subseteq B\subseteq G$ defined over $\CO_F$. Let $\Phi$ be the set of roots of $T$. Let $V = X_\ast(T)\otimes\BR$. We write our Borel and its opposite as $B = TU$ and $B^-=TU^-$ of $G$ over $\CO_F$. Let $\Phi^+$ and $\Phi^-$ be the corresponding sets of positive roots and negative roots respectively. Let $X_\ast(T)^+$ be the corresponding set of dominant characters. Let $W_0 = N_G(T)(L)/T(L)$ be the finite Weyl group. Let $\tilde\BS$ be the index set of affine simple reflections. In other words, the set of affine simple reflections is $\{s_i\mid i\in \tilde\BS\}$. Let $\BS_0$ be the index set of finite simple reflections. For each $i\in\BS_0$, let $\a_i$ and $\a_i^{\vee}$ be the corresponding simple root and coroot respectively. 

We denote by $L^+$ the positive loop group functor, i.e.\ $L^+ G(R) = G(R\doublebracket t)$ if $F$ has equal characteristic resp.\ the Witt vector version $L^+G(R) = G(W(R)\otimes_{W(\mathbb F_q)}\mathcal O_F)$ if $F$ is of mixed characteristic (cf.\ \cite[Section~1.1]{Zhu}).
We define the loop group functor $L$ similarly: In the equal characteristic case, $LG(R) = G(R\doubleparenthesis t)$ is an ind-scheme over $\CO_F$, whereas in the mixed characteristic case, $L(R) = G(W(R)\otimes_{W(\mathbb F_q)}F)$ seen as ind-(perfect scheme).

As Iwahori subgroup $\CI\subseteq G(L)$, we choose the preimage of $B^-(\overline{\mathbb F}_q)$ under the natural projection $G(\mathcal O_L)\to G(\overline{\mathbb F}_q)$. As in \cite[Section~1.4]{Zhu}, this is set of $\mathcal O_L$-rational points of an Iwahori group scheme. The affine flag variety is the fpqc quotient of $LG$ by this Iwahori group scheme (cf.\ \cite[Section~9]{BS}), and it is an ind-(perfect scheme) over $\overline{\mathbb F}_q$. The set of $\overline{\mathbb F}_q$-rational points is given by $G(L)/\CI$.


For the remaining paper, we will only consider perfect schemes in the mixed characteristic setting and suppress the adjective \enquote{perfect}, so our notion of \enquote{scheme} coincides with the usual definition only in the equal characteristic setting.

As scheme over $\BF_q$, the ring $\CO$ or, more precisely, the positive loop group of the one-dimensional additive group $L^+G_a$, is isomorphic to the infinite dimensional affine space $\mathbb A^{\mathbb N}$. We fix once and for all an embedding $\iota : \mathbb A^1\hookrightarrow L^+ G_a$ such that after taking $\overline{\BF}_q$-valued points, the composite map 
\begin{align*}
    \overline{\BF}_q\xrightarrow{\iota(\overline{\BF}_q)} \CO_L\xrightarrow{\varepsilon\to 0}\overline{\BF}_q
\end{align*}
is the identity. In the case of equal characteristic, $\iota(\overline{\BF}_q)$ can be chosen to be the map sending $x\in\overline{\BF}_q$ to the constant power series $x\in L$. In mixed characteristic, the usual choice is to send $x\in \overline{\BF}_q$ to its Teichm\"uller lift in $L$.

For each $\a \in \Phi$, we have the root subgroup $U_\a$. Note that $\s$ sends $U_\a$ to $U_{\s(\a)}$. Fix an isomorphism $u_\a:LG_a\to U_\a$ for each root $\a$ and fix a lifting $\dot{z} \in N_G(T)(\CO_L)$ for each $z\in W_0$ in the usual compatible way as in \cite[\S  2]{HL15}. More explicitly, we require the following.
\begin{enumerate}[(i)]
    \item For any reduced expression $z = s_1 s_2\cdots s_k$, we have $\dot{z} = \dot{s}_1\dot{s}_1 \cdots \dot{s}_k$.
    

    \item For $i\in \BS_0$, we have $\dot{s}_{i}^2 =  \a_i^{\vee}(-1) $.
    \item  For $i\in \BS_0$ and $x\in \bar{\BF}_p$, we have $\dot{s}_i^{-1} u_{\a_i}(x) \dot{s}_i = u_{-\a_i}(-x)$ 

    \item For $\a\in \Phi$ and $x\in \bar{\BF}_p$, we have $\s(u_{\a}(x)) = u_{\s(\a)}(\s(x))$.
\end{enumerate}


Let $\tW = N_G(T)(L)/T(\CO_L) $ be the Iwahori-Weyl group. Note that we can identify $\tW$ with $X_*(T)\rtimes W_0=\{t^\l z \mid \l\in X_*(T),z\in W_0\}$. For any $w = t^{\l}z\in \tW$, we fix a lifting $\dot{w} =  \l(t)\dot{z} \in N_G(T)(\CO_L)$ extending the lifting of $W_0$.

Let $\Phi_{\af} =  \BZ \times \Phi$ be the set of affine roots. By convention, the set of affine simple roots are
$$\D_{\af} = \{ a_i = (0, -\a_i) \mid i\in \BS_0 \}    \sqcup \{a_0 = (1,\theta)\},$$
where $\th$ is the highest root. Let $\Phi^+_{\af} = (\BZ_{\ge0} \times \Phi^-) \sqcup (\BZ_{\ge1} \times \Phi^+ )$. Let $\Phi^-_{\af} = \Phi_{\af} - \Phi^+_{\af}$.

For any $a  = (k,\a) \in \BZ\times \Phi$, define the map of $\overline{\BF}_q$-schemes
$u_a:\mathbb A^1 \to LU_\a$ via the composition
\begin{align*}
    \mathbb A^1 = G_a\xrightarrow\iota L^+ G_a \hookrightarrow L G_a\xrightarrow{\cdot \varepsilon^k} LG_a\xrightarrow{u_\a}U_\a.
\end{align*}
In more explicit terms, this sends $x\in \overline{\BF}_q$ to $u_\a(\iota(x)\varepsilon^k)$.
We denote the image of $u_a$ by $U_a\subseteq L U_\a$, the affine root subgroup corresponding to the affine root $a = (k,\a)$.
The corresponding affine reflection is $s_a = t^{k\a^{\vee}} s_\a$.

\subsection{}
For any $b\in G(L)$, denote by \begin{align*}[b] = \{g^{-1}b\sigma(g)\mid g\in G(L)\}
\end{align*}be the $\s$-conjugacy class of $G(L)$ containing $b$. Denote $B(G) =\{[b] \mid b\in G(L)\}.$

Let $b \in G(L)$ and $w\in \tW$. Define the affine Deligne-Lusztig variety (ADLV) $X_w(b)$ to be the reduced subscheme of the affine flag variety with $\overline{\mathbb F}_q$-valued points
$$X_w(b) = \{ g\CI\in G(L)/\CI\mid g^{-1}b\s(g) \in \CI \dot{w} \CI   \}.$$

Note that the isomorphism class of $X_w(b)$ depends only on $w$ and the $\s$-conjugacy class $[b]$ of $b$. Note also that $X_w(b)$ is naturally acted by the group 
$$\BJ_b(F) = \{ g \in G(L) \mid  g^{-1}b\s(g) = b\},$$
the group of $\s$-centralizer of $b$.

\subsection{}
One fundamental method of studying affine Deligne-Lusztig varieties in affine flag varieties is the \emph{Deligne-Lusztig reduction method}, as introduced (in the affine case) by Görtz-He.

\begin{proposition}[{Cf.\ \cite[Corollary 2.5.3]{Goertz2010b}}]\label{prop:DL-reduction}
Let $w\in \tW$ and let $i\in \tilde\BS$.
The following two statements hold for any $b\in G(\breve F)$.
\begin{enumerate}[(1)]
\item If $\ell(s_i w \s (s_{ i}))=\ell(x)$, then there exists a $\mathbb J_b(F)$-equivariant universal homeomorphism $X_w(b)\rightarrow X_{s_i w \s(s_i)}(b)$.
\item If $\ell(s_i w \s(s_i))=\ell(x)-2$, then there exists a decomposition $X_w(b)=X_{I}\sqcup X_{\text{II}}$ such that
\begin{itemize}
\item $X_{I}$ is open and there exists a $\mathbb J_b(F)$-equivariant morphism $X_{I}\rightarrow X_{s_i w}(b)$, which is  the composition of a Zariski-locally trivial $\mathbb G_m$-bundle and a universal homeomorphism. 
\item $X_{\text{II}}$ is closed and there exists a $\mathbb J_b(F)$-equivariant morphism $X_{\text{II}}\rightarrow X_{s_i w \s(s_i)}(b)$, which is the composition of a Zariski-locally trivial $\mathbb A^1$-bundle and a universal homeomorphism. 
\end{itemize}
\end{enumerate}
\end{proposition}

By iterating this result, one can subdivide every $X_w(b)$ into locally closed pieces, which are iterated fibre bundles over smaller ADLV $X_{w_1}(b),\ldots, X_{w_N}(b)$ with $\mathbb A^1$ and $\mathbb G_m$-fibres. These \enquote{smaller} ADLV are easier to handle.

\begin{definition}\label{def:min-len}
We say that $w\in \tW$ is a \emph{minimal length element}, or is \emph{min-len}, if $\ell(w)\leq \ell(u^{-1} w\s(u))$ for all $u\in \tW$. 
\end{definition}

For $w, w' \in \tW$ and $i \in \tilde \BS$, we write $w \xrightarrow{s_i}_\s w'$ if $w'=s_i w \s(s_i)$ and $\ell(w') \le \ell(w)$. We write $w \to_\s w'$ if there is a sequence $w=w_0, w_1, \ldots, w_n=w'$ of elements in $\tW$ such that for any $k$, $w_{k-1} \xrightarrow{s_i}_\s w_k$ for some $s_i \in \tilde \BS$. We write $w \approx_\s w'$ if $w \to_\s w'$ and $w' \to_\s w$.

\begin{theorem}[{\cite[Theorem~A]{HN14}}]\label{thm:minlen}
For every $w\in \tW$, there exists a minimal length element $w' \in \tW$ such that $w \to_\s w'$.
\end{theorem}
Theorem \ref{thm:minlen} together with Proposition~\ref{prop:DL-reduction} means that we can decompose every ADLV $X_w(b)$ in to locally closed pieces, which are iterated fibre bundles over ADLV $X_{w_1}(b),\ldots,X_{w_N}(b)$ such that each $w_i$ is of minimal length. 

\subsection{}\label{sec:2.4}
Let $w \in \tW$. We construct the reduction tree for $w$, which encodes the Deligne-Lusztig reduction for the affine Deligne-Lusztig varieties associated with $w$.

The reduction tree is constructed inductively. 

Suppose that $w$ is min-len, then the reduction tree of $w$ consists of a single vertex $w$ and no edges. 

Suppose that $w$ is not min-len and that a reduction tree is given for any $z \in \tW$ with $\ell(z)<\ell(w)$. By Theorem~\ref{thm:minlen}, there exist $w' \in \tW$ and $a \in \D_{\af}$ with $w \approx_\s w'$ and $s_a w' \s(s_a)<w'$. The reduction tree of $w$ is the graph containing the given reduction tree for $ w' \s(s_a)$, the reduction tree for $s_a w' \s(s_a)$, and the edges $w \lup  w' \s(s_a)$ and $w \lup s_a w' \s(s_a)$. 

Note that the reduction trees of $w$ are not unique. They depend on the choices of $w'$ and $s_a$ in each step of the construction.

Let $\CT$ be a reduction tree of $w$. An {\it end point} of the tree $\CT$ is a vertex $x$ of $\CT$ such that there is no edge of the form $x \lup x'$ in $\CT$. By Theorem~\ref{thm:minlen}, each end point is min-len. A {\it reduction path} in $\CT$ is a path $\underline p: w \lup w_1 \lup \cdots \lup w_n$, where $w_n$ is an end point of $\CT$. The {\it length} $\ell(\underline p)$ is the number of edges in $\underline p$. We also write $\operatorname{end}(\underline p)=w_n$ and $[b]_{\underline p}=  [ \dot{w}_n ] \in B(G)$.

Note that if $x \lup y$, then $\ell(x)-\ell(y) \in \{1, 2\}$. We say that the edge $x \lup y$ is of type I if $\ell(x)-\ell(y)=1$ and  of type II if $\ell(x)-\ell(y)=2$. For any reduction path $\underline p$, we denote by $\ell_{\text{I}}(\underline p)$ the number of type I edges in $\underline p$ and by $\ell_{\text{II}}(\underline p)$ the number of type II edges in $\underline p$. Then $\ell(\underline p)=\ell_{\text{I}}(\underline p)+\ell_{\text{II}}(\underline p)$. 

For any $a_1, a_2 \in \BN$, we say that a scheme $X$ is an {\it iterated fibration} of type $(a_1, a_2)$ over a scheme $Y$ if there exist morphisms $$X=Y_0 \to Y_1 \to \cdots \to Y_{a_1+a_2}=Y$$ such that for any $i$ with $0 \le i<a_1+a_2$, $Y_i$ is a Zariski-locally trivial $\BA^{1}$-bundle or $\BG_m$-bundle over $Y_{i+1}$, and with $\BG_m$ occurring $a_1$ times and $\BA^1$ occurring $a_2$ times.

Combining Proposition~\ref{prop:DL-reduction} with the construction of the reduction trees, we obtain the following result. 

\begin{proposition}\label{prop:dec}
Let $w \in \tW$ and $\CT$ be a reduction tree of $w$. Then, for any $b \in G(L)$, there exists a decomposition $$X_w(b)=\bigsqcup_{\underline p \mathrm{\; is\; a\; reduction\; path\; of\; } \CT,     \operatorname{end}(\ud p)=[b]} X_{\underline p},$$ where $X_{\underline p}$ is a locally closed subscheme of $X_w(b)$, and is $\BJ_b(F)$-equivariant universally homeomorphic to an iterated fibration of type $(\ell_{\text{I}}(\underline p), \ell_{\text{II}}(\underline p))$ over  $X_{\operatorname{end}(\underline p)}(b)$.
\end{proposition}

For any $w\in\tW$, denote $B(G)_w = \{[b] \in B(G)\mid X_w(b) \ne \emptyset\}$. Then by Proposition \ref{prop:dec}, we have 
$$B(G)_w = \{ [b]_{\ud p}\} \mid \underline p \mathrm{\; is\; a\; reduction\; path\; of\; } \CT;  \text{end}(\ud p) \in [b] \},$$
where $\CT$ is some (or any) reduction tree of the given $w$.

\section{Lifting Deligne-Lusztig Reduction}\label{sec:3}

\subsection{Fundamentals on liftings}

\begin{definition}
Let $Y$ be a locally closed subvariety of $G(L)/\CI$. A \emph{lift} (or a \emph{global section}) of $Y$ is a morphism $\psi:Y\to G(L)$ of (perfect) schemes over $\overline{\BF}_q$ such that $\pi \circ \psi = \id$, where $\pi$ is the natural projection $G(L)\to G(L)/\CI$.
\end{definition}

We give some simple facts about liftings of affine Deligne-Lusztig varieties. 

\begin{lemma}\label{lem:lift}
Let $w \in \tW$ and $b\in G(L)$. Suppose $X_w(b)$ has a lift $\psi:X_w(b) \to G(L)$.
\begin{enumerate}
    \item Let $h \in G(L)$ and $b' = h^{-1}  b \s(h)$. Then 
$g\CI \mapsto h^{-1} \psi(h g\CI)$ is a lift of $X_w(b')$.

\item Let $j \in \BJ_b(F)$. Then $g\CI \to j\psi(j^{-1} g\CI)$ is a lift of $X_w(b)$.

\end{enumerate}
\end{lemma}

\begin{lemma}\label{lem:liftAffineSchubertCell}
    Let $w\in \tW$. Then the affine Schubert cell $\CI \dot{w} \CI\subseteq G(L)/\CI$ has a lift.
\end{lemma}
\begin{proof}
    Induction on $\ell(w)$. If $\ell(w)=0$, then $\CI \dot{w} \CI/\CI = \dot w\CI/\CI$ is just a point, which has a lift $\dot w\CI/\CI\mapsto \dot w\in G(L)$.

    In the inductive step, consider a simple affine index $i\in \tilde\BS$ with $\ell(s_i w) = \ell(w)-1$. Write the simple affine root as $a_i = (k,\alpha)$ with $k\in\mathbb Z,\alpha\in \Phi$ and consider the affine root subgroup $U_a\subseteq LG$ defined over $\BF_q$. Set $v = s_i w$. Note that each element $g\in \CI \dot w \CI/\CI$ can be written in the form $g = u_a(g_1) \dot s_a g_2$ with uniquely determined $g_2\in \CI \dot v \CI/\CI$ and $g_1\in \overline{\BF}_q$. Moreover, this defines an algebraic map $p_a: \CI \dot w \CI/\CI \to \mathbb A^1$ over $\overline{\BF}_q$, sending $g$ to $g_1$.

    By the inductive assumption, we get a lift $\psi_v : \CI \dot{v} \CI/\CI \to G(L)$. Then we can define a lift $\psi_w : \CI \dot{w} \CI/\CI \to G(L)$ sending $g\CI$ to $u_a(p_a(g\CI)) \psi_v(u_a(p_a(g\CI))^{-1} g\CI)$. This finishes the induction and the proof.
\end{proof}

Let $b \in G(L)$, $w\in \tW$ and $s = s_i$ ($i\in\tilde\BS$).  In this section, we shall study the following three special cases.
\begin{itemize}
    \item \textbf{Case 0.} $\ell( w) = \ell(sw\s(s))$ and $X_{sw\s(s)}(b)$ has a lift.

    \item \textbf{Case 1.} $\ell(sw\s(s)) = \ell(w) -2$ and $X_{sw}(b)$ has a lift.

    \item \textbf{Case 2.} $\ell(sw\s(s)) = \ell(w) -2$ and $X_{sw\s(s)}(b)$ has a lift.
    
\end{itemize}
We shall prove that in these cases, the lift can be ``upgraded'' along reduction path and moreover, in Case 1 and Case 2, the fiber bundles mentioned in Proposition \ref{prop:DL-reduction} are trivial bundles (see Proposition \ref{prop:type0}, \ref{prop:type1}, \ref{prop:type2}).

Following \cite[\S2.4]{GH10}, for $g\CI, g'\CI \in G(L)/\CI , w\in \tW$, we write $g\CI \xrightarrow{w} g'\CI   $ if $g\CI$ and $g'\CI$ have relative position $w$.


\subsection{}
In this subsection, we study case 0. Let $b \in G(L)$, $w\in \tW$ and $s = s_a$ ($a\in\D_{\af}$) such that $\ell(w) = \ell(sw\s(s))$. Assume that $X_{sw\s(s)}(b)$ has a lift $\psi:X_{sw\s(s)}(b) \to G(L)$. Our goal is to construct a lift $\Psi$ of $X_w(b)$. We assume that $w\neq sw\s(s)$.

We first consider the case $sw <w$. We have the following identification
$$X_w(b) = \left\{ g\CI \xrightarrow{w} b\s(g)\CI \right\}  \xlongrightarrow{\sim} 
\left\{\begin{tikzcd}
g\CI \arrow{r}{w} \arrow{d}{s} & b\sigma(g)\CI\arrow{d}{\s(s)} \\
h\CI \arrow{ur}{sw}\arrow{r}{s w\s(s)} & b\s(h)\CI
\end{tikzcd}
\right\}. $$
Define $\g : X_{w}(b) \to X_{s w \s(s)}(b)$ sending $g\CI$ to $h\CI$. This is a universally homeomorphism (see \cite[Corollary 2.5.3]{GH10}). 


Recall that we have the affine root subgroup $U_a$ for all affine root $a$ (see \S\ref{sec:2.1}).

For any $h\CI\in X_{sw\s(s)}(b)$, denote $\tilde{h} = \psi(h\CI)$. 

We define a lift $\Psi$ of $X_w(b)$ as follows. Let $g\CI \in X_w(b)$. Let $h\CI = \g(g\CI)$. Note that $\tilde{h}^{-1}g \in \CI\dot{s}\CI$. Let $z $ be the unique element in $U_a$ such that $\tilde{h}^{-1} g \in z\dot{s} \CI$, or equivalently, $g\CI  = \tilde{h} z \dot{s} \CI$. Define $\Psi(g\CI) = \tilde{h}z\dot{s} \in G(L)$. Then this is a lift of $X_w(b)$.

Next, we consider the case $w\s(s) < w$. We have the following identification
$$X_w(b) = \left\{ g\CI \xrightarrow{w} b\s(g)\CI \right\}  \xlongrightarrow{\sim} 
\left\{\begin{tikzcd}
g\CI \arrow{r}{w} \arrow{dr}{w\s(s)} & b\sigma(g)\CI  \\
h\CI \arrow{u }{s }\arrow{r}{s w\s(s)} & b\s(h)\CI \arrow{u}[swap]{\s(s)}
\end{tikzcd}
\right\}.$$

Unlike the previous case, the map $g\CI\mapsto h\CI$ is not algebraic. Consider the algebraic map $g\CI \mapsto b\s(h)\CI \mapsto b\s(\tilde{h}) \in G(L)$. Note that $g^{-1} b \s(\tilde{h}) \in \CI \dot{w}\s(\dot{s}) \CI$. There is a unique element $z \in \prod_{ a'\in \Phi^{+}_{\af}, w\s(s)(a')\in \Phi^-_{\af}}U_{a'} $ such that $ g^{-1} b \s(\tilde{h}) \in  \CI \dot{w}\s(\dot{s})  z $, or equivalently, $g\CI = b\s(\tilde{h}) z^{-1} \s(\dot{s})^{-1} \dot{w}^{-1} \CI$. Here, the product of $U_{a'}$ is independent of the order. Then define $\Psi(g \CI) = b\s(\tilde{h}) z^{-1} \s(\dot{s})^{-1} \dot{w}^{-1}$. This is a lift of $X_w(b)$.

In summary, we have proved that
\begin{proposition}\label{prop:type0}
Let $w\in \tW$, $b\in G(L)$ such that $X_w(b)\ne\emptyset$. Let $s = s_a$ ($a\in\D_{\af}$) such that $\ell( w) = \ell(s w \s(s))$. Assume that $X_{sw \s(s)}(b)$ has a lift $\psi$. Then $X_w(b)$ has a lift.
\end{proposition}

\subsection{}\label{sec:3.3}
In this subsection, we study case 1. Let $ b \in G(L)$, $w\in \tW$ and $s = s_a$ ($a\in\D_{\af}$) such that $\ell(s w \s(s)) = \ell(w) -2$. We identify $X_{w}(b)$ with the set
$$ \left\{\begin{tikzcd}
g\CI \arrow{r}{w} \arrow{d}{s} & b\sigma(g)\CI\arrow{d}{\s(s)} \\
  h\CI \arrow{ur}{sw}  & b\s(h)\CI
\end{tikzcd}   \right\} = X_{\text{I}} \sqcup X_{\text{II}},$$
where $X_{\text{I}}$ is the set of diagrams
\begin{equation}\tag*{(3.1)}\label{eq:3.1}
\begin{tikzcd}
g\CI \arrow{r}{w} \arrow{d}{s} & b\sigma(g)\CI\arrow{d}{\s(s)} \\
  h\CI \arrow{ur}{sw} \arrow{r}{sw}  & b\s(h)\CI
\end{tikzcd}   
\end{equation}
and $X_{\text{II}}$ is the set of diagrams
\begin{equation}\tag*{(3.2)}\label{eq:3.2}
\begin{tikzcd}
g\CI \arrow{r}{w} \arrow{d}{s} & b\sigma(g)\CI\arrow{d}{\s(s)} \\
  h\CI \arrow{ur}{sw} \arrow{r}{sw\s(s)} & b\s(h)\CI.
\end{tikzcd}     
\end{equation}

Assume that $X_{sw }(b)$ has a lift $\psi $. Our goal is to construct a lift $\Psi$ of $X_{\text{I}}$ and prove that the natural projection $X_{\text{I}}\to X_{sw}(b)$ is a trivial fibration (up to a universal homeomorphism).

For any $h\CI\in X_{sw }(b)$, denote $\tilde{h} = \psi(h\CI)$. 

Let $g\CI \in X_{\text{I}}$ be represented by the diagram \ref{eq:3.1}. Note that $\tilde{h}^{-1}g \in \CI\dot{s}\CI$. Let $z$ be the unique element in $U_a$ such that $\tilde{h}^{-1}g \in z\dot{s}\CI $, or equivalently, $g\CI = \tilde{h} z \dot{s}\CI$. Define $\Psi(g\CI) = \tilde{h}z\dot{s} \in G(L)$. This is a lift of $X_{\text{I}}$.

Let $X_{\text{I}}'$ be the set of diagrams
\begin{equation}\tag*{(3.3)}\label{eq:3.3}    
\begin{tikzcd}
  & g'\CI\arrow{d}{\s(s)} \\
  h\CI \arrow{ur}{sw} \arrow{r}{sw}  & b\s(h)\CI
\end{tikzcd}
\end{equation}
The natural projection $X_{\text{I}} \to X_{\text{I}}'$ is a universal homeomorphism. We shall define $\g: X_{\text{I}}' \to X_{sw}(b)\times \BG_m$ and $\g': X_{sw}(b)\times \BG_m \to  X_{\text{I}}' $. Recall that we have the morphism $u_a:\overline{\BF}_q \to U_\a(L)$ (See \S\ref{sec:2.1}).

Note that $\tilde{h}^{-1}b\s(\tilde{h}) \in \CI \dot{s} \dot{w} \CI$. Let $\varphi(\tilde{h})$ be the unique element in $\BA^1$ such that $\tilde{h}^{-1}b\s(\tilde{h})u_{\s(a)}(\varphi(\tilde{h}))\s(\dot{s}) \in \CI s w \s(s) \CI $.

Given a diagram of the form \ref{eq:3.3}. Note that $(b\s(\tilde{h})) ^{-1} g' \in \CI \s(\dot{s})\CI$. Let $y$ be the unique element in $\BA^1$ such that $(b\s(\tilde{h})) ^{-1} g' \in u_a(y) \s(\dot{s}) \CI$, or equivalently, $g'\CI = b\s(\tilde{h}) u_a(y)\s(\dot{s})\CI$. The image of the diagram under $\g$ is defined to be the pair $(h\CI , y - \varphi(\tilde{h}))$.

Let $h\CI \in X_{sw}(b)$ and $x \in \BG_m$. Define $\g'(h\CI,x)$ be the diagram
$$\begin{tikzcd}
 & b \s (\tilde{h}) u_{\s(a)}(x + \varphi(\tilde{h})) \s(\dot{s})  \CI \arrow{d}{\s(s)} \\
  h\CI \arrow{ur}{sw} \arrow{r}{sw } & b\s(h)\CI.
\end{tikzcd}  
$$

Using the definition of $\varphi(\tilde{h})$, it is straightforward to check that $\g$ and $\g'$ are well-defined. Using the relation $g'\CI = b\s(\tilde{h}) u_a(y)\s(\dot{s})\CI$, it is easy to verify that $\g$ and $\g'$ are inverse to each other.

In summary, we have proved the following.

\begin{proposition}\label{prop:type1}
Let $w\in \tW$, $[b]\in B(G)_w$ and $s = s_a$ ($a\in \D_{\af}$) such that $ \ell(w) -2 = \ell( s  w \s(s ) )$. Assume that $X_{sw}(b)$ has a lift. Then $X_{\text{I}}$ has a lift and the natural map $X_{\text{I}} \to X_{sw}(b)$ is a composition of a universal homeomorphism and a trivial $\BG_m$ fibration.
\end{proposition}

\subsection{}\label{sec:3.4}
In this subsection, we study case 2. Let $b\in G(L)$, $w\in \tW$ and $s = s_a$ ($a\in\D_{\af}$) such that $\ell(s w \s(s)) = \ell(w) -2$. Let $X_{\text{I}}$ and $X_{\text{II}}$ be as in the previous subsection.

Assume that $X_{sw \s(s)}(b)$ has a lift $\psi $. Our goal is to construct a lift $\Psi$ of $X_{\text{II}}$ and prove that the natural projection $X_{\text{II}}\to X_{sw \s(s)}(b)$ is a trivial fibration. We shall define $\g: X_{\text{II}} \to X_{sw \s(s)}(b)\times \BA^1$ and $\g': X_{sw \s(s)}(b)\times \BA^1 \to  X_{\text{II}} $.

For any $h\CI\in X_{sw \s(s) }(b)$, denote $\tilde{h} = \psi(h\CI)$. 

Given a diagram of the form \ref{eq:3.2}. Note that $\tilde{h}^{-1}g \in \CI \dot{s}\CI$. Let $y $ be the unique element in $\BA^1$ such that $\tilde{h}^{-1}g \CI = u_a(y )\dot{s}\CI$, or equivalently, $g \CI = \tilde{h} u_a(y  )\dot{s}\CI$. The image of the diagram under $\g$ is defined to be $ ( h\CI , y)   $. Moreover, the map $\Psi: g\CI \mapsto \tilde{h} u_a(y) \dot{s}  $ is a lift of $X_{\text{II}}$.

Let $h\CI \in X_{sw\s(s)}(b)$ and $x \in \BA^1$. Define $\g'(h\CI,x)$ be the diagram
$$
\begin{tikzcd}
\tilde{h}u_a(x) \dot{s} \CI \arrow{r}{w} \arrow{d}{s} & b \s (\tilde{h}u_a(x) \dot{s}  ) \CI \arrow{d}{\s(s)} \\
  h\CI \arrow{ur}{sw} \arrow{r}{sw\s(s)} & b\s(h)\CI.
\end{tikzcd}  
$$

It is straightforward to check that $\g$ and $\g'$ are well-defined. Using the relation $g \CI =  \tilde{h}  u_a(y) \dot{s} \CI$, it is easy to verify that $\g$ and $\g'$ are inverse to each other.

In summary, we have proved the following.

\begin{proposition}\label{prop:type2}
Let $w\in \tW$, $[b]\in B(G)_w$ and $s = s_a$ ($a\in \D_{\af}$) such that $ \ell(w) -2 = \ell( s  w \s(s ) )$. Assume that $X_{sw \s(s)}(b)$ has a lift. Then $X_{\text{II}}$ has a lift and the natural projection $X_{\text{II}} \to X_{s w \s(s)}(b)$ is a trivial $\BA^1$ fibration.
\end{proposition}

\section{Main Result}\label{sec:4}
\subsection{}
Let $w \in \tW$ and $[b]\in B(G)_w$. Let $\CT$ be a reduction tree of $w$ (see \ref{sec:2.4}).

We now state the main result of this paper.
\begin{theorem}\label{thm:main}
Let $w\in\tW$, $b \in G(L)$ with $X_w(b)\ne\emptyset$. Let $\mathcal{T}$ be a reduction tree of $w$. Assume that
\begin{enumerate}
    \item there exists a unique path $\ud p$ of $\CT$ such that $\text{end}(\ud p) \in [b]$.
    \item $X_{\mathrm{end}(\ud p)}(b) $ has a lift.
\end{enumerate}
Then $X_w(b)$ is universally homeomorphic to the trivial fiber bundle
$$  X_{\mathrm{end}(\ud p)}(b) \times (\BG_m)^{\ell_{\text{I}}(\ud p)} \times  (\BA^1)  ^{\ell_{\text{II}}(\ud p)},$$
where $\ell_{\text{I}}(\ud p)$ and $\ell_{\text{II}}(\ud p)$ are the number of type I and type II reduction steps in $\ud p$ respectively.
\end{theorem}

\begin{proof}
Using induction, the theorem follows from Proposition \ref{prop:type0}, \ref{prop:type1} and \ref{prop:type2}.
\end{proof}

We point out that Theorem \ref{thm:main} is independent of the choice of the reduction tree $\CT$. We say $(w,[b],\CT)$ satisfies the strong multiplicity one condition, if there exists a unique path $\ud p$ of $\CT$ such that $\text{end}(\ud p) \in [b]$. Using the theory of class polynomial (cf. \cite[Theorem 6.7]{HN14}), one can prove that $(w,[b],\CT)$ satisfies the strong multiplicity one condition for some reduction tree $\CT$ of $w$, if and only if $(w,[b],\CT')$ satisfies the strong multiplicity one condition for any reduction tree $\CT'$ of $w$. Moreover, one can prove that the number $\ell_{\text{I}}(\ud p)$ and $\ell_{\text{I}}(\ud p)$ in Theorem \ref{thm:main} is independent of the choice of the reduction tree.

\section{Geometric Coxeter Type Elements}\label{sec:5}

\subsection{}\label{sec:5.1}
\begin{definition}
Let $(\CW,\CS)$ be a Coxeter system together with a length preserving automorphism $\d$. 
\begin{enumerate}
    \item An element $c\in \CW$ is called a \textit{$\d$-Coxeter element}, if $c$ is a product of simple reflections, one from each $\d$-orbit of $\CS$.

    \item An element $c\in \CW$ is called a \textit{partial $\d$-Coxeter element}, if $c$ is a product of simple reflections, at most one from each $\d$-orbit of $\CS$.

\end{enumerate}
\end{definition}

Extend the length function $\ell$ on $\tW$ to $\tW\rtimes \<\s\>$ naturally with $\ell(\s)=0$. An element $w\in \tW$ is call $\s$-straight, if $\ell( (w\s)^k ) = k\ell(w )$ for all $k \in \BZ_{>0}$, or equivalently, $(\CI w \s \CI )^k = \CI (w \s)^k \CI\subseteq G(L)\rtimes\<\s\>$.  

For a subset $K\subseteq \tilde{\BS}$, denote by ${}^{K}\tW^{\s(K)}$ the set of minimal representatives of double cosets. For $x\in \tW$, we write $x\s(K) = K$ if $x \in {}^{K}\tW^{\s(K)}$ and $\{x\s(s_i)x^{-1} \mid i \in K\} = \{s_j \mid j\in K\} $. The following important result is proved in \cite{HN14}. 
\begin{theorem}[{\cite[Theorems 3.4]{HN14}}]
Let $w\in \tW$. Then there exists $K \subseteq \tilde{\BS}$ with $W_K$ finite, a $\s$-straight element $x \in \tW$ with $x \in {}^{K}\tW^{\s(K)}$, $x\s(K) =K$ and $u \in W_K$ such that $w \to_\s ux$ and $ux$ is a minimal length element.
\end{theorem}

\begin{definition}\label{def:min-cox}
Let $w\in \tW$ be a minimal length element. We say $w$ is a \textit{minimal Coxeter type element}, if there exists $K \subseteq \tilde{\BS}$ with $W_K$ finite, a $\s$-straight element $x \in \tW$ with $x \in {}^{K}\tW^{\s(K)}$ and $x\s(K) =K$ and a $\Ad(x)\circ\s$-Coxeter element $c_K \in W_K$ such that $w \approx_\sigma c_Kx$.
\footnote{
One can also use the theory of $P$-alcove elements developed in \cite[\S 3]{GHKR10}, \cite[\S3.3]{GHN1} and \cite[Lemma 5.15]{SSY} to formulate the notion of minimal Coxeter type elements. But it will be more complicated and we do not use this approach. }
\end{definition}

Here are some simple examples.
\begin{example}\label{ex:geo-cox}
\begin{enumerate}
\item If $w$ is itself a $\s$-Coxeter in some $\s$-stable finite $W_K$, then $w$ is automatically of minimal Coxeter type. Indeed, take $x = 1$ and $c_K = w$ in Definition \ref{def:min-cox}.

\item A length-zero element $\t$ in $\tW$ is of minimal Coxeter type. Indeed, take $K = \emptyset$, $c_K =1$ and $x = \t$ in Definition \ref{def:min-cox}.

\item In the case of type $A_5$. $w=s_4 \t_3 = \t_3s_1$ is a minimal Coxeter type element. Indeed, take $x = \t_3$, $K = \{1,4\}$, $c_K =s_4 $ in Definition \ref{def:min-cox}.

\item In the case of type $C_2$. $w=s_1 \tau_2$ is a minimal Coxeter type element. Indeed, take $x = \t_2$, $K = \{1 \}$, $c_K =s_1 $ in Definition \ref{def:min-cox}.

\item In the case of type ${}^{2}A_4$. $w=s_1 \t_1 = \t_1s_0$ is a minimal Coxeter type element. Indeed, take $x = \t_1$, $K = \{1,0\}$, $c_K =s_1 $ in Definition \ref{def:min-cox}.

\end{enumerate}
\end{example}

\begin{definition}\label{def:geo-cox}
We say $w\in\tW$ is of \textit{geometric Coxeter type} if for some (or any) reduction tree $\CT$ of $w$,
\begin{enumerate}
\item \text{strong multiplicity one} holds, i.e., for any $[b]\in B(G)_w$, there is a unique path $\ud p$ of $\CT$ such that $\mathrm{end}(\ud p)\in [b]$;
\item any endpoint $e$ of $\CT$ is a minimal Coxeter type element.
\end{enumerate}

\end{definition}

Elements of the form $t^{\mu}c \in {}^{\BS_0}\tW$, where $\mu\in X_*(T)^+$ and $c$ a partial $\s$-Coxeter element have been studied extensively in \cite{HNY1}. By \cite[Theorem 2.6]{HNY1}, such elements $t^{\mu}c$ are of geometric Coxeter type.

In \cite{SSY}, the authors generalize the elements $t^{\mu}c$ to ``positive Coxeter type" elements. By \cite[Theorem 5.7]{SSY}, such elements are of geometric Coxeter type.

The notion of geometric Coxeter type elements is broader than the above two notions. Indeed, in Example \ref{ex:geo-cox} (4), the classical part of $w$ is $s_1s_2s_1s_2$, which is not even a $\s$-conjugate of a partial $\s$-Coxeter element.\footnote{This happens when $K$ is not a very special subset of $J_{\af}$ with respect to $\tau_J$. (cf. \cite[\S5.3]{SSY})}

The following special case is of independent interest. 
\begin{theorem}[{\cite[Theorems 5.12]{SSY}}]
Let $w \in \tW$ be a minimal length element. Suppose the classical part of $w$ is $\s$-conjugate to a partial $\s$-Coxeter element. Then $w$ is of minimal Coxeter type. In particular, in the case of type (split) $A_n$, all minimal element are minimal Coxeter type elements.
\end{theorem}
We will not use this theorem in the rest of the paper.

\subsection{}\label{sec:5.2} In this section, we study structure of affine Deligne-Lusztig variety $X_w(b)$ for minimal Coxeter type element $w$.

Let $w = c_K x$ be a minimal length element. Here, $K$ is a subset of $\tilde{\BS}$, $W_K$ finite, $x$ is $\s$-straight, $x \in {}^{K}\tW^{\s(K)}$, $x\s(K) =K$ and $c_K$ is a $\Ad(x)\circ\s$-Coxeter element in $W_K$. Let $b = \dot{x}$, then $B(G)_w$ contains only $[b]$, the $\s$-conjugacy class containing $\dot{x}$.

Let $\CK$ be the standard parahoric subgroup of $G(L)$ corresponding to $K$. Note that $\CK$ is $\Ad(\dot{x})\circ\s$-stable. Let $\bar{\CK}$ be the reductive quotient of $\CK$ and let $\bar{\CI}$ be the image of $\CI$ in $\bar{\CK}$. For $h\in \CK$, denote by $\bar{h}$ the image of $h$ in $\bar{\CK}$. Since $K$ is $\Ad(\t_J)\circ\s$-stable, the map $\bar{\s}: \bar{h}\mapsto \overline{\dot{\t}_J \s(h) \dot{\t}_J^{-1}}$ is a Frobenius morphism on $\bar{\CK}$. Note that $\bar{\CI}$ is a $\bar{\s}$-stable Borel subgroup of $\bar{\CK}$. Define
\begin{align*}
    X_w^{\CK}(b) &= \{ g \in \CK/\CI \mid g^{-1} b \s(g) \in \CI \dw \CI\},\\
    X_{c_K}^{\bar{\CK}} &= \{ \bar{h} \in \bar{\CK}/\bar{\CI} \mid \bar{h}^{-1}\bar{\s} (\bar{h})   \in \bar{\CI} \dot{c}_K \bar{\CI}\}.
\end{align*}
Recall that $w = c_Kx$ and $b = \dx$. By the proof of \cite[Theorem 4.8]{He14}, we have
\begin{enumerate}[(a)]
    \item the multiplication map induces a bijection 
    \begin{align*}\tag*{(5.1)}\label{eq:5.1}
    \BJ_b(F)\times_{\BJ_b(F) \cap \CK} X_w^{\CK}(b) \xlongrightarrow{\sim}  X_w(b),
    \end{align*}
    
    \item the reduction map induces an isomorphism 
    \begin{align*}\tag*{(5.2)}\label{eq:5.2}
    X_w^{\CK}(b) \xlongrightarrow{\sim} X_{c_K}^{\bar{\CK}}.
    \end{align*}
\end{enumerate}
Note that $X_{c_K}^{\bar{\CK}}$ is a classical Deligne-Lusztig variety of the reductive group $\bar{\CK}$ corresponding to the $\bar{\s}$-Coxeter element $c_K$. In particular, all irreducible components of $X_w(b)$ are in a single $\BJ_b(F)$-orbit, and each irreducible component is isomorphic to a classical Deligne-Lusztig variety of Coxeter type, which is known to be affine.

Let $w_{0,K}$ be the longest element in the finite Weyl group $W_K$. By a classical result of Lusztig \cite[Theorem 2.6]{Lusz76}, we get $X_{c_K}^{\bar{\CK}}\subseteq \bar{\CI} w_{0,K} \bar{\CI}/\bar{\CI}$. Hence $X_{c_K}^{\CK}\subseteq \CI \dot{w}_{0,K} \CI/\CI$. By Lemma~\ref{lem:liftAffineSchubertCell}, we know that all affine Schubert cells are liftable, hence $X_{c_K}^{\CK}$ has a lift $\psi_{\CK}$.

We now extend the lift $\psi_{\CK}$ to a lift $\psi$ of $X_w(b)$ via \ref{eq:5.1} as follows. Fix a representative $j$ for each $[j] \in \BJ_b(F)/ \BJ_b(F) \cap \CK$. Then, \ref{eq:5.1} implies 
$$X_w(b) = \bigsqcup_{[j]\in \BJ_b(F)/ \BJ_b(F)\cap \CK} jX_w^{\CK}(b).$$ 
For $g\CI \in j X_w^{\CK}(b)$ for some $[j]\in \BJ_b(F)/ \BJ_b(F)\cap \CK$, define $\psi( g \CI) = j \psi_{\CK}( j^{-1}g\CI)$. The definition of the lift $\psi$ depends on the choices of the representatives $j$ (so our proof uses the axiom of choice).

By Lemma \ref{lem:lift} (1), for any $b'$ with $X_w(b')\ne\emptyset$, the variety $X_w(b')$ has a lift.

\subsection{} The main result of this section is the following.
\begin{theorem}\label{thm:geo-cox}
Let $w\in \tW$ be a geometric Coxeter type element and $b\in G(L)$ with $X_w(b) \ne \emptyset$. Then 
\begin{enumerate}
    \item all irreducible components of $X_w(b)$ lie in a single $\BJ_b(F)$-orbit;

    \item each irreducible component of $X_w(b)$ is universally homeomorphic to the trivial bundle
$$X'\times (\BG_m)^{\ell_{\text{I}}(\ud p)} \times (\BA^1) ^{\ell_{\text{II}}(\ud p)},$$
where $X'$ is a classical Deligne-Lusztig variety of Coxeter type, $\ell_{\text{I}}(\ud p)$ and $\ell_{\text{II}}(\ud p)$ are as in Theorem \ref{thm:main}.
\end{enumerate}
\end{theorem}

\begin{proof}
By Definition \ref{def:geo-cox} (1), condition (1) of Theorem \ref{thm:main} holds. By Definition \ref{def:geo-cox} (2), Proposition \ref{prop:type0} and the results in \S\ref{sec:5.2}, condition (2) of Theorem \ref{thm:main} also holds. Then the theorem following from Theorem \ref{thm:main}.
\end{proof}

\section{The Purity Property}\label{sec:6}
\subsection{}
The purity of the Newton stratification is an important geometric property. For us, it provides us with information about the set $B(G)_w$ and a lower bound of the dimension of $X_w(b)$ for $w\in \tW$.

By \cite[Theorem 7.4 (iii)]{Chai}, the poset
$B(G)$ is ranked, i.e., for any $[b_1], [b_2] \in B(G)$ such that $[b_1] \le [b_2]$, any two maximal chains from $[b_1]$ to $[b_2]$ in $B(G)$ have the same length. We denote the common length by $\ell([b_1], [b_2])$.

For any $[b]\in B(G)$, let $\nu(b)$ be the Newton vector of $[b]$ (cf. \cite{Kot85}). The defect of $[b]$ is defined as
\begin{align*}
    \de(b) = \rank_F(G) - \rank_F (\BJ_b).
\end{align*}
By \cite[Theorem 7.4 (iv)]{Chai}, \cite[Proposition 3.11]{Ham15} and \cite[Theorem 3.4]{Vie20}, we have 
\begin{align*}
\ell([b_1],[b_2]) = \<\nu(b_2)-\nu(b_1),\rho\>+\frac{1}{2}\de(b_1)-\frac{1}{2}\de(b_2).
\end{align*}

\begin{definition}[{cf. \cite[\S 2.1]{HNY2}}]
    For $[b_1],[b_2]\in B(G)$ with $[b_1]\leq [b_2]$, we define the \emph{essential gap} to be
\begin{align*}
    \text{Ess.-gap.}([b_1], [b_2]) &= \langle \nu(b_2) - \nu(b_1),\rho\rangle -\frac 12 \de(b_1) + \frac 12\de(b_2) \\ 
    &= \ell([b_1], [b_2]) - \de(b_1) +\de(b_2).
\end{align*}
\end{definition}

Let $w \in \tW$. It is known from Viehmann that the set $B(G)_w$ always has a unique minimum \cite{Vi21} and maximum \cite{Viehmann2014}. We denote them by $[b_{w,\min}]$ and $[b_{w,\max}]$.

\begin{proposition}[{\cite[Proposition 3.1]{HNY2}}]\label{prop:ess-gap}Let $w\in \tW$ and $[b]\in B(G)_w$, we have 
$$\dim X_w(b)\ge \dim X_w(b_{w,\max}) +\text{Ess-gap}([b],[b_{w,\max}])  . $$
\end{proposition}

The following theorem characterize the case where equality in Proposition \ref{prop:ess-gap} holds.
\begin{theorem}[Cf.\ {\cite[Theorem 3.16]{MV20}}]\label{thm:purity}
    Let $w \in \tW$ and $[b] \in B(G)_w$. The following are equivalent:
    \begin{enumerate}
        \item $\dim X_w(b') - \dim X_w(b_{w,\max}) \leq \text{Ess.-gap}([b']
    , [b_{w,\max}])$ for all $[b'] \in B(G)_x$ with $[b] \leq [b']$.
        \item $\dim X_w(b') - \dim X_w(b_{w,\max}) = \text{Ess.-gap}([b']
    , [b_{w,\max}])$ for all $[b'] \in B(G)_x$ with $[b] \leq [b']$.
        \item If $[b'] \in B(G)$ with $[b] \leq [b'] \leq [b_{w,\max}]$, then $[b'] \in B(G)_x$, and the closure of $\CI \dw \CI \cap [b']$ inside $\CI \dw \CI$ is the union of all $\CI \dw \CI \cap [b'']$ for $[b''] \in B(G)_w$ with $[b''] \leq [b']$.
    \end{enumerate}
\end{theorem}



Our main result of this section will be the following:

\begin{theorem}\label{thm:gctPurity}
Let $w$ be a geometric Coxeter type element. Then
\begin{enumerate}
    \item the set $B(G)_w$ is \emph{saturated}, that is, \begin{align*} B(G)_w =\{ [b]\in B(G) \mid  [b_{w,\min}] \le [b] \le [b_{w,\max}]\};\end{align*}
    \item for any $[b] \in B(G)_w$, we have $\dim X_w(b) + \<\nu_b,2\rho\> = \ell(w) - \ell([b], [b_{\max}])$:
   
    \item for any $[b]\in B(G)_w$, the closure of the Newton stratum $\CI \dot{w} \CI\cap [b]$ inside $\CI \dot{w} \CI$ is the union of all Newton strata $\CI \dot{w} \CI\cap [b']$ for $[b'] \in B(G)_w$ with $[b'] \leq [b]$.
\end{enumerate}

\end{theorem}

\subsection{}
For the proof of Theorem \ref{thm:gctPurity}, we need some preparation. We write $\cl : \tW\to W_0$ for the natural projection, sending $w = t^{\mu} z\in \tW$ to its classical part $\cl\, w = z\in W_0$. Moreover, we write $\ell_{R,\sigma}: W_0\to \mathbb Z_{\geq 0}$ for the $\sigma$-twisted reflection length from \cite[Section~4]{SSY}.



\begin{lemma}\label{lem:reflectionLength}
    Let $K\subsetneq \tilde\BS$ be a spherical subset and $x\in {}^{K} \tW {}^{\sigma (K)}$ with $\Ad(x)\circ\sigma$ normalizing $W_K$. Let $u\in W_K$ and $w = ux$. Then $\ell_{R,\sigma}(\cl\, w) = \ell_{R,\sigma}(\cl\, x) + \ell_{R,x\circ\sigma}(u)$.
\end{lemma}
\begin{proof}
    Denote by $\Phi_K$ the finite subroot system of $\Phi_{\af}$ generated by the $a_i$ for $i\in K$. We write $W_K\subseteq \tW$ for its finite Weyl group.

    Let $V = X_\ast(T)_{\Gamma_0}\otimes\mathbb Q$. 
    Throughout this proof, we use $\cl\, w\sigma$ to refer to the endomorphism of $V$ given by the composite action of the finite Weyl group element $\cl\,w\in W_0$ and the Frobenius $\sigma$. Consider the map
    \begin{align*}
        \varphi := \mathrm{avg}_{K}\vert_{V^{ \cl\, w\sigma }} : V^{\cl\, w\sigma}\rightarrow V,~v\mapsto \frac 1{\# W_{K}}\sum_{g\in W_K} (\cl\,g)v.
    \end{align*}

    Denote by $V_K$ the $\mathbb Q$-subspace of $V$ generated by the coroots of the classical parts $(\cl\,a_i)^\vee$ for $i\in K$.
    Then $v-\varphi(v)\in V_K$ for all $v\in V^{\cl\, w\sigma}$. We see that $\mathrm{ker}(\varphi) = V^{\cl\, w\sigma}\cap V_K$. Thus
    \begin{align*}
        \dim \mathrm{ker}(\varphi) &= \dim V_K^{\cl(ux)\sigma } \\&= \dim V_K^{\cl\,x\sigma } - \ell_{R,\cl\,x\sigma }(\cl\,u)
        \\&=\dim V_K^{\cl\,x\sigma} - \ell_{R, x\sigma}(u)
    \end{align*} Moreover, one checks easily for $v\in V^{\cl\, w\sigma}$ that
    \begin{align*}
        \cl\, w\sigma\varphi(v) &= \frac 1{\# W_K} \sum_{g\in W_K} \sigma  \cl(ux\sigma(g))\sigma v \\&= \frac 1{\# W_{K}} \sum_{g'\in W_{K}} (\cl g)\cl (ux)\sigma v
        \\&\underset{v\in V^{\cl\, w\sigma}}=\frac 1{\# W_J} \sum_{g'\in W_J} (\cl g)v = \varphi(v).
    \end{align*}
    We claim that
    \begin{align*}
        \mathrm{im}(\varphi) = \{v\in V^{\cl\, w\sigma}\mid \langle v,\cl\, a_i\rangle = \{0\}\forall i\in K\}.
    \end{align*}Indeed, the the inclusion $\mathrm{im}(\varphi)\subseteq V^{\cl\, w\sigma}$ has just been proved, and it is clear that $\langle v,\cl\, a_i\rangle=\{0\}$ for all $v\in \mathrm{im}(\varphi), i\in K$ by definition of $\varphi$. It remains to see that the restriction of $\varphi$ to the claimed image is the identity map to obtain $\supseteq$.
    
    We conclude
    \begin{align*}
        &\dim V^{\cl\, w\sigma} = \dim \mathrm{ker}(\varphi) + \dim\mathrm{im}(\varphi) \\&=  \dim V_K^{\cl\,x\sigma} - \ell_{R,x\sigma}(u) + \dim \{v\in V\mid \cl \,w\sigma (v) = v\text{ and }\langle v,a_i\rangle = \{0\}\forall i\in K\}.
    \end{align*}
    Repeating the same calculation for the special case $u=1$, we get
    \begin{align*}
        \dim V^{\cl\, x\sigma} = \dim V_K^{\cl\, x\sigma} + \dim \{v\in V\mid \cl\,x\sigma (v) = v\text{ and }\langle v,a_i\rangle = \{0\}\forall i\in K\}.
    \end{align*}
    If $v\in V$ satisfies $\langle v,a_i\rangle=\{0\}$ for all $i\in K$, then certainly $\cl\,w\sigma v = \cl\,x\sigma v$. Hence we conclude
    \begin{align*}
        \dim V^{\cl\,w\sigma} - \dim V^{\cl\,x\sigma} = \ell_{R,x\sigma}(u).
    \end{align*}
    In other words, $\ell_{R,\sigma}(\cl\, w) = \ell_{R,\sigma}(\cl\, x) + \ell_{R,\sigma  x}(u)$. This finishes the proof.
\end{proof}
\begin{proposition}\label{prop:mctCharacterization}
    Let $w\in \tW$, denote the $\sigma$-conjugacy class of a representative $\dot w\in N_G(T)(L)$ by $[\dw]\in B(G)$. Then we have
    \begin{align*}
        \ell(w) \geq \langle \nu([\dw]),2\rho\rangle + \ell_{R,\sigma}(\cl\, w) - \de([\dw]).
    \end{align*}
    Moreover, equality holds if and only if $w$ is of minimal Coxeter type.
\end{proposition}
\begin{proof}
    Note that the right-hand side of the inequality is invariant under $\sigma$-conjugation of $w$. So we may assume WLOG that $w$ is min-len and of the form $w = ux$ such that there is a spherical subset $K\subsetneq \tilde\BS$ with $u\in W_K$.

    By the above lemma, we get $\ell_{R,\sigma}(\cl\,w) = \ell_{R,\sigma}(\cl\, w) + \ell_{R,\Ad(x)\circ\sigma}(u)$. Moreover, it is clear that $\ell(w) = \ell(x)+\ell(u)$.

    Note that $\kappa(x) = \kappa(w)$ and for some $n>0$, the $\sigma$-twisted power $x^{\sigma,n} = x\sigma(x)\cdots \sigma^{n-1}(x)$ agrees with $w^{\sigma,n}$. Thus $[w] = [x]\in B(G)$. 

    From the definition of straight elements, we get
    \begin{align*}
        \ell(x) = \langle \nu([\dot x]),2\rho\rangle = \langle \nu([\dot w]),2\rho\rangle.
    \end{align*}
    Moreover, the identity $\de(\dot w]) = \de([\dot x]) = \ell_{R,\sigma}(\cl\, x)$ follows from \cite[Proposition~3.9]{Sch23}. So we conclude
    \begin{align*}
        &\ell(w) -( \langle \nu([\dw]),2\rho\rangle + \ell_{R,\sigma}(\cl\, w) - \de([\dw]))
        \\&=\ell(u) - \ell_{R,\Ad(x)\circ\sigma}(u).
    \end{align*}
    This quantity is always non-negative, and equal to zero if and only if $w$ is of minimal Coxeter type.
\end{proof}
\begin{remark}
    In the statement of Proposition~\ref{prop:mctCharacterization}, we may simply write
    \begin{align*}
        \langle \nu([\dw]),2\rho\rangle - \de([\dw]) = \langle \lambda([\dw]),2\rho\rangle
    \end{align*}
    by \cite[Proposition~3.9]{Sch23}.
\end{remark}

\begin{proposition}\label{prop:gctDimensionFormula}
    Let $w\in \tW$ be an element of geometric Coxeter type and let $[b]\in B(G)_w$. Then $(w,b)$ satisfies the equivalent conditions of Theorem~\ref{thm:purity} and
    \begin{align*}
        \dim X_w(b) = \frac 12\left(\ell(w)  + \ell_{R,\sigma}(\cl\, w) - \langle \nu(b),2\rho\rangle - \de(w)\right).
    \end{align*}
\end{proposition}
\begin{proof}
    We prove the statement via induction on $\ell(w)$.

    For the inductive start, we consider the case where $w$ is of minimal Coxeter type. Then $[b]\in B(G)_w$ implies $[b] = [\dw] = [b_{w,\max}]$ and $\dim X_w(b) = \ell(w) - \langle \nu(b),2\rho\rangle$. The condition (2) of Theorem~\ref{thm:purity} is trivially satisfied in this case. Moreover, writing $\ell(w)$ as in Proposition~\ref{prop:mctCharacterization}, we get the claimed dimension formula.
    
    Next, we assume that $w$ is not of minimal Coxeter type and that the proposition has been proved for all geometric Coxeter type elements of smaller length than $w$. Note that our statement is invariant under length-preserving cyclic shifts $w \approx_{\s} sw\sigma(s)$. Using the Deligne-Lusztig reduction, we hence may assume that there exists a simple affine reflection $s$ with $\ell(sw\sigma(s))<\ell(w)$.
    
    By the strong multiplicity one property, the set $B(G)_w$ is the disjoint union of $B(G)_{sw}$ and $B(G)_{sw\sigma(s)}$. We study these subsets separately.
    \begin{enumerate}[(i)]
    \item
    First consider the case $[b]\in B(G)_{sw}$. Then the strong multiplicity one property shows
    \begin{align*}
        \dim X_w(b) &= 1+\dim X_{sw}(b) \underset{\text{ind}}= \frac 12\left(\ell(sw) + \ell_{R,\sigma}(\cl\,sw) - \langle \nu(b),2\rho\rangle - \de(b)\right)
        \intertext{Noticing $\ell(sw) = \ell(w)-1$ and $\ell_{R,\sigma}(\cl\,sw) \geq \ell_{R,\sigma}(\cl\,w)-1$, we get}
        \cdots&\geq \frac 12\left(\ell(w) + \ell_{R,\sigma}(\cl\,w) - \langle \nu(b),2\rho\rangle - \de(b)\right).\tag{i.1}
    \end{align*}
    By the Deligne-Lusztig reduction method, we get $[b_{w,\max}] = [b_{sw,\max}]$. Hence
    \begin{align*}
        &\dim X_w(b) - \dim X_{w}(b_{w,\max}) = (1+\dim X_{sw}(b)) - (1+\dim X_{sw}(b_{sw,\max}))
        \\&~\underset{\text{ind}}=\text{Ess.-gap}([b], [b_{w,\max}]).\tag{i.2}.
    \end{align*}
    \item Now we consider the case $[b]\in B(G)_{sw\sigma(s)}$. Then the strong multiplicity one property shows
    \begin{align*}
        \dim X_w(b) &= 1+\dim X_{sw\sigma(s)}(b) \\&\underset{\text{ind}}= \frac 12\left(\ell(sw\sigma(s)) + \ell_{R,\sigma}(\cl\,sw\sigma(s)) - \langle \nu(b),2\rho\rangle - \de(b)\right).
        \intertext{Noticing $\ell(sw\sigma(s)) = \ell(w)-2$ and $\ell_{R,\sigma}(\cl\,sw\sigma(s)) = \ell_{R,\sigma}(\cl\,w)$, we get}
        \cdots&= \frac 12\left(\ell(w) + \ell_{R,\sigma}(\cl\,w) - \langle \nu(b),2\rho\rangle - \de(b)\right).\tag{ii.1}
    \end{align*}
    \end{enumerate}
    Now we prove that property (1) of Theorem~\ref{thm:purity} is satisfied for all $[b]\in B(G)_w$. Indeed, for $[b]\in B(G)_{sw}$, this is (i.2). For $[b]\in B(G)_{sw\sigma(s)}$, this follows from the dimension formula in (ii.1) together with the estimate (i.2), the latter one applied to $[b_{w,\max}]$. This shows that $(w,b)$ satisfies the equivalent conditions of Theorem~\ref{thm:purity} for all $[b]\in B(G)_w$.

    Now consider $[b_1]\in B(G)_{sw}$ and pick an arbitrary $[b_2]\in B(G)_{sw\sigma(s)}$. Then using the formula from Theorem~\ref{thm:purity} (2), we get
    \begin{align*}
        \dim X_w(b_1) &= \dim X_w(b_{w,\max}) + \text{Ess.-gap}([b_1], [b_{w,\max}]) \\&= 
        \dim X_w(b_2) + \text{Ess.-gap}([b_1], [b_{w,\max}]) - \text{Ess.-gap}([b_2], [b_{w,\max}]).
    \end{align*}
    Using the explicit formula for the essential gap together with (ii.1) applied to $\dim X_w(b_2)$, we get the claimed formula for for $X_w(b_1)$. This finishes the induction and the proof.
\end{proof}


It is clear that Proposition~\ref{prop:gctDimensionFormula} together with Theorem~\ref{thm:purity} implies Theorem~\ref{thm:gctPurity}.

Finally, one may ask to describe the lengths of paths and endpoints for geometric Coxeter elements $w\in\tW$ and arbitrary $[b]\in B(G)$ with $X_w(b)\neq\emptyset$, similar to \cite[Theorem~7.1]{HNY1} or \cite[Theorem~5.7, Proposition~6.7]{SSY}. For the path lengths, this is easy.
\begin{theorem}\label{thm:lengthExplicit}
Let $w\in \tW$ be an element of geometric Coxeter type and $[b]\in B(G)$ with $X_w(b)\neq\emptyset$. Then every reduction path for $w$ and $[b]$ contains precisely $\ell_1$ edges of type I, and $\ell_2$ edges of type II, where
\begin{align*}
    \ell_1 &:= \# I(\nu(b_{w,\min}))/\sigma - \# I(\nu(b))/\sigma = \# I(\nu(\dot w))/\sigma - \# I(\nu(b))/\sigma,
    \\\ell_2&:=\ell([b], [b_{w,\max}]) = \frac 12\left(\ell(w) - \ell_{R,\sigma}(\cl\, w) - \langle \nu(b),2\rho\rangle + \de(b)\right).
\end{align*}
In particular, as scheme over $\overline{\mathbb F}_q$, the ADLV $X_w(b)$ is universally homeomorphic to
\begin{align*}
    \mathbb G_m^{\ell_1}\times \mathbb A^{\ell_2} \times C \times D,
\end{align*}where $C$ is a classical Deligne-Lusztig variety of Coxeter type and $D$ is a union of closed points.
\end{theorem}



For the proof, we need the following lemma.
\begin{lemma}\label{lem:lengthExplicitHelper}
    Let $w\in\tW$ be an element satisfying strong multiplicity one and consider a simple affine reflection $s = s_i$ for $i\in \tilde\BS$ with $\ell(sw\sigma(s)) = \ell(w)-2$. Put $w_1 = sw$ and $w_2=sw\sigma(s)$. Assume that $B(G)_w, B(G)_{w_1}$ and $B(G)_{w_2}$ are saturated. Then
    \begin{enumerate}[(1)]
    \item $[b_{w,\min}] = [b_{w_2,\min}]$,
    \item $I(\nu(b_{w,\min}))\subsetneq I(\nu(b_{w_1,\min}))$ and $I(\nu(b_{w_1,\min})) - I(\nu(b_{w_2,\min}))$ is equal to the $\sigma$-orbit of a simple root and
    \item $[b_{w,\max}] = [b_{w_1,\max}]$.
    \end{enumerate}
\end{lemma}
\begin{proof}
    Part (3) follows from the Deligne-Lusztig reduction method, noticing that the $\sigma$-conjugates of $\CI w_1 \CI$ contain a non-empty open subset of $\CI w \CI$, which hence must intersect $[b_{w,\max}]$.
    
    Then part (1) follows from formal arguments: By the Deligne-Lusztig reduction method, we get $B(G)_w = B(G)_{w_1} \cup B(G)_{w_2}$, and this union is disjoint by the strong multiplicity one property. If we had $[b_{w_1,\min}]\leq [b_{w_2,\min}]$, we would get $B(G)_{w_2}\subseteq B(G)_{w_1}$ by saturatedness and (1). Hence $[b_{w,\min}] = \min([b_{w_1,\min}], [b_{w_2,\min}])$ agrees with $[b_{w_2,\min}]$.

    By \cite[Theorem~1.1]{Vi21} and \cite[Corollary~4]{Sch24_strata}, we get 
    \begin{align*}
        \nu(b)\equiv \nu(b_{w,\min}) \pmod{I(\nu(b_{w,\min}))}
    \end{align*} for all $[b]\in B(G)_w$, and similarly for $w_1$ and $w_2$. Using saturatedness, we get
    \begin{align*}
        B(G)_w = \{[b]\in B(G)\mid [b]\leq [b_{w,\max}]\text{ and }\nu(b) \equiv \nu(b_{w,\max})\pmod{I(\nu(b_{w,\min}))}\},
    \end{align*}
    and similar for $w_1, w_2$.
    
    Applying this to $[b] = [b_{w_1,\min}]$, we get $I(\nu(b_{w_1,\min}))\subsetneq I(\nu(b_{w,\min}))$. Define
    \begin{align*}
        J := I(\nu(b_{w,\min}))\setminus I(\nu(b_{w_1,\min})).
    \end{align*}
    For any $\sigma$-orbit $o\subseteq J$, we find a $\sigma$-conjugacy class $[b_o]$ whose Kottwitz point agrees with that of $[b_{w_1,\min}]$ and whose Newton point is the $I(\nu(b_{w_1,\min}))\cup o$-average of the Newton point $\nu(b_{w_1,\min})$. Since $[b_{w,\min}]\leq [b_o]<[b_{w_1,\min}]$, we get $[b_o]\in B(G)_{w_2}$. By saturatedness, $B(G)_{w_2}$ contains the join of all these $\sigma$-conjugacy classes. One may easily see that this join is $[b_{w_1,\min}]$ if $J$ contains at least two distinct $\sigma$-orbits of simple roots. This is a contradiction to strong multiplicity one. Hence $J$ contains exactly one such orbit, i.e.\ $J$ is a $\sigma$-orbit of simple roots. This proves (2).
\end{proof}
\begin{proof}[Proof of Theorem~\ref{thm:lengthExplicit}]
    We first prove that the two expressions for $\ell_{\text{II}}$ agree. Using the length formula
    \begin{align*}
        \ell([b], [b_{w,\max}]) = \langle \nu(b_{w,\max}) - \nu(b),\rho\rangle + \frac 12(\de(b) - \de(b_{w,\max}))
    \end{align*}
    from \cite{Chai}, it suffices to show this in case $[b] = [b_{w,\max}]$. Then we use Proposition~\ref{prop:gctDimensionFormula} together with \cite[Theorem~2.23]{He-CDM} to get
    \begin{align*}
        &\frac 12\left(\ell(w) + \ell_{R,\sigma}(\cl\, w) -\langle \nu(b_{w,\max}),2\rho\rangle - \de(b_{w,\max})\right) = \dim X_w(b_{w,\max}) \\&= \ell(w) - \langle \nu(b_{w,\max}),2\rho\rangle.
    \end{align*}
    So the value of $\ell_{\text{II}}$ is well-defined in all cases.

    We show that $[b_{w,\min}] = [\dot w]$ and that every reduction path $\underline p$ of $w$ and $[b]$ satisfies $\ell_{\text{I}}(\underline p) = \ell_{\text{I}}$ and $\ell_{\text{II}}(\underline p) = \ell_{\text{II}}$. Both claims are proved simultaneously using induction on $\ell(w)$.

    For the inductive start, assume that $w$ is a minimal length element. The claim $[b_{w,\min}] = [\dot w]$ follows from \cite[Theorem~3.5]{He14}. Now $\ell_{I} = \ell_{\text{II}}=0$ as $[b] = [\dot w] = [b_{w,\max}]$ by loc.\ cit.

    In the inductive step, we observe that our claims are invariant under length preserving cyclic shifts $w\leftrightarrow sw\sigma(s)$. So regarding the Deligne-Lusztig reduction method, it suffices to focus on the case where $\ell(sw\sigma(s))=\ell(w)-2$ for a simple affine reflection $s$. In this case, the Deligne-Lusztig reduction method together with strong multiplicity one show that $B(G)_w$ is the disjoint union $B(G)_{w_1}\sqcup B(G)_{w_2}$, where $w_1 = sw$ and $w_2 = sw\sigma(s)$. Note that the inductive assumption applies to both $w_1$ and $w_2$.

    In particular, we get
    \begin{align*}
        [b_{w,\min}]\underset{\text{L\ref{lem:lengthExplicitHelper}}}= [b_{w_2,\min}]\underset{\text{ind}}=[\dot w_2] = [\dot w]
    \end{align*}
    as $w$ and $w_2$ are $\sigma$-conjugate. It remains to prove the claim on $\ell_{\text{I}}(p)$ and $\ell_{\text{II}}(p)$.

    First consider the case where $[b]\in B(G)_{w_1}$. Then $\underline p$ must be the type I edge $w\to w_1$ followed by a path $\underline p_1$. We get
    \begin{align*}
        \ell_{\text{I}}(p) &= 1+\ell_{\text{I}}(p_1) \underset{\text{ind}}=1 + \# I(\nu(b_{w_1,\min}))/\sigma - \#I(\nu(b))/\sigma \\&\underset{\text{L\ref{lem:lengthExplicitHelper}}}=\#I(\nu(b_{w,\min}))/\sigma - \# I(\nu(b))/\sigma\\
        \ell_{\text{II}}(p)&=\ell_{\text{II}}(p_1) \underset{\text{ind}}=\ell([b], [b_{w_1,\max}]) \underset{\text{L\ref{lem:lengthExplicitHelper}}}=\ell([b], [b_{w, \max}]).
    \end{align*}
    It remains to consider the case where $[b]\in B(G)_{w_2}$. Then $\underline p$ must be the type II edge $w\to w_2$ followed by a path $\underline p_2$. We get
    \begin{align*}
        \ell_{I}(p) &= \ell_{I}(p_2) \underset{\text{ind}}= \# I(\nu(b_{w_2},\min))/\sigma - \# I(\nu(b))/\sigma\\&\underset{\text{L\ref{lem:lengthExplicitHelper}}}=\# I(\nu(b_w,\min))/\sigma - \# I(\nu(b))/\sigma\\
        \ell_{\text{II}}(p)&=1+\ell_{\text{II}}(p_2)\underset{\text{ind}}=1+\ell([b], [b_{w_2,\max}]).
    \end{align*}
    It hence suffices to show that $1 = \ell([b_{w_2,\max}], [b_{w,\max}])$. From our earlier proof of \enquote{well-definedness of $\ell_{\text{I}}$}, we get
    \begin{align*}
        \ell([b_{w_2,\max}], [b_{w,\max}]) = \frac 12\left(\ell(w) - \ell(w_2) + \ell_{R,\sigma}(\cl\, w) -\ell_{R,\sigma}(\cl\, w_2)\right).
    \end{align*}
    Noticing that $\ell(w_2) = \ell(w)-2$ by assumption and that $\cl\, w$ is $\sigma$-conjugate to $\cl\, w_2$, this expression evaluates to $1$ as desired. This finishes the induction and the proof.
\end{proof}

\printbibliography
\end{document}